\documentclass[a4paper,12pt]{amsart}
\usepackage[T1]{fontenc}
\usepackage[utf8]{inputenc}
\usepackage[english]{babel}
\usepackage{amsmath,amssymb,amsthm,mathabx,verbatim}
\usepackage[all]{xy}
\usepackage[bookmarks=true]{hyperref}
\usepackage{color}
\usepackage{graphicx}
\usepackage{cancel}
\usepackage{datetime}
\usepackage{tikz}
\usepackage{booktabs}
\usepackage[shortlabels]{enumitem}

\usepackage[text={170mm,230mm},centering]{geometry}

{\left\lbrace\begin{array}{@{}l@{}}}%
{\end{array}\right.}

\newcommand{\ZZ}{\ensuremath{\mathbb{Z}}}

\newcommand{\parity}{\mathop{\mathrm{par}}}

\theoremstyle{definition}
\newtheorem{definizione}{Definition}[section]
\theoremstyle{plain}

\theoremstyle{plain}
\newtheorem{teorema}[definizione]{Theorem}
\theoremstyle{plain}
\newtheorem{corollarioth}[definizione]{Corollary}
\theoremstyle{plain}
\newtheorem{proposizione}[definizione]{Proposition}
\theoremstyle{plain}

\theoremstyle{plain}

\theoremstyle{remark}
\newtheorem{remark}[definizione]{Remark}
\theoremstyle{remark}
\newtheorem{example}[definizione]{Example}

\begin{document}

\frenchspacing

\title{Partition identities associated to Rogers-Ramanujan type identities}

\author{Pietro Mercuri}
\email{mercuri.ptr@gmail.com}

\thanks{The author has been supported by the research grant ``Ing. Giorgio Schirillo'' of the Istituto Nazionale di Alta Matematica ``F. Severi'', Rome.}

\date{}


\begin{abstract}
We show that, in many cases, there are infinitely many sets of partitions corresponding to a single analytical Rogers-Ramanujan type identity. This means that a single analytical Rogers-Ramanujan type identity implies the existence of bijections among infinitely many sets of partitions. We also give an explicit description of these infinite sets coming from the sum side of the analytical identity explaining how to interpret the sum side combinatorially as the generating function of the partitions considered. Moreover, we give a new infinite familiy of Rogers-Ramanujan type identities obtained by the Glaisher's identities.
\end{abstract}

\maketitle

\section{Introduction}

The analytical Rogers-Ramanujan type identities are identities between an infinite product and an infinite series that express the same power series in one unknown usually denoted by $q$. The name arise after the work \cite{RR19} of Rogers and Ramanujan on analytical identities of this kind. In the last few decades a renewed interest about this kind of identities has arisen since the works of Lepowsky and Wilson (\cite{LW84},\cite{LW85}) that linked some of these identities to representation theory of Lie algebras and vertex operator algebras. Their approach was developed by many authors (see for example \cite{Cap93}, \cite{Cap96}, \cite{KR19}, \cite{MP87}, \cite{MP99}, \cite{MP01}, \cite{Pri94}, \cite{PS16}, \cite{Sil17}) leading to many new interesting links and unknown Rogers-Ramanujan type identities. Both sides of many analytical Rogers-Ramanujan type identities have a simple combinatorial theoretical interpretation in terms of partitions of integers. In these cases the analytical identity implies the existence of a bijection between different sets of partitions corresponding to the two sides. See for example \cite{Cap04}, \cite{Hir79}, \cite{SA88} and \cite{Sub85}. In some cases there have been generalizations and refinements of the partition identities that arose from analytical Rogers-Ramanujan type identities. See for example \cite{AAG95}, \cite{And92}, \cite{Dou14}, \cite{Dou17}, \cite{DL18} and \cite{DL19}. On the other hand, each side of an analytical Rogers-Ramanujan type identity could have many different combinatorial interpretations. For instance, Agarwal in \cite{Aga96} shows the equinumerosity of three different sets of ($n$-color) partitions having the same generating function. In this paper we show that certain classes of generating functions admit infinitely many partition theoretical interpretations.

An analytical Rogers-Ramanujan type identity, as we said, implies the existence of bijections between the sets of partitions associated to the two sides of the identity, but a simple description of these bijections may be hard to be determined. See \cite{AAG95}, \cite{CDMV20} and \cite{GM81} as examples. Also, the opposite may happen: There is a simple description of a bijection between two sets of partitions, but it is hard to find an analytical identity of Rogers-Ramanujan type. See \cite{KR19} as an example. Glaisher in \cite{Gla83} generalized the Euler's bijection corresponding to the classical Euler's identity. He gave an infinite family of new partition identities, but we don't know of any analytical Rogers-Ramanujan type identity associated to the partition identities found by Glaisher. In this paper we show an infinite family of new analytical Rogers-Ramanujan type identities. Each analytical identity of the family is associated to one of the Glaisher's identity.

The paper has the following structure: In Section \ref{sec:ineq} we set up notation and introduce a convenient way for describing some kind of conditions on partitions. See \cite{And84} for a deeper and wider reference about the theory of integer partitions. In Section \ref{sec:results} we explain how infinitely many interpretations can be given to a single generating function of a certain type. This leads to infinite classes of partition identities. We also give some explicit examples of them. In Section~\ref{sec:bij} we show that there is a simple explicit bijection between each pair of partition sets achievable by Section~\ref{sec:results}. We use this to get a simple explicit bijection for the partition identity associated to the second classical Rogers-Ramanujan identity. The bijection given preserves the number of parts but not the weight of the partitions. In Section~\ref{sec:euler} we apply the ideas of Section~\ref{sec:results} to give a new interpretation of the classical Euler's identity. Moreover, we give an analytical Rogers-Ramanujan type identity for each Glaisher's identity providing an infinite new class of Rogers-Ramanujan type identities. In the Appendix, we express the partition identities found in \cite{Cap04}, \cite{Hir79}, \cite{SA88} and \cite{Sub85} using the notation introduced in Section \ref{sec:ineq}.

\section{Notation and chains of inequalities} \label{sec:ineq}

For $n\in\ZZ$, we define
\[
(q)_n:=\begin{cases}
\prod_{s=1}^n (1-q^s), & \text{if } n>0, \\
1, & \text{if } n\le 0.
\end{cases}
\]
Let $n,N\in\ZZ_{>0}$, we can write a generic partition of $N$ with at most $n$ parts using a \emph{chain of inequalities} as
\begin{equation}\label{eq:generic partition}
p_1 \ge p_2\ge p_3\ge \ldots\ge p_{n-1}\ge p_n \ge 0,
\end{equation}
where $N=p_1+p_2+\ldots +p_n$ and $N$ is called the \emph{weight} of the partition $(p_1,p_2,\ldots,p_n)$.
\begin{remark}\label{rem:pov}
Let $\mathcal{P}_{\le n}$ be the set of the partitions with parts not bigger than $n$ and let $\mathcal{P}_{\le n}^*$ be the set of the conjugate partitions of $\mathcal{P}_{\le n}$, i.e., partitions with at most $n$ parts. We recall that $\frac{1}{(q)_n}$ can be easily seen as the generating function of the partitions in $\mathcal{P}_{\le n}$ interpreting the coefficient of $q^s$ as the number of parts equal to $s$. Since $\#\mathcal{P}_{\le n}=\#\mathcal{P}_{\le n}^*$, then $\frac{1}{(q)_n}$ can be also seen in terms of the conjugate partitions interpreting the coefficient of $q^s$ as the difference of the $s$-th part from the $(s+1)$-th part, where we consider the $(n+1$)-th part to be $0$. In this work, we use this latter point of view.
\end{remark}
We can write a partition with exactly $n$ distinct parts using the chain of inequalities:
\begin{equation}\label{eq:distinct parts}
a_1>a_2>a_3>\ldots >a_{n-1}>a_n\ge 1.
\end{equation}
We can get the chain (\ref{eq:distinct parts}) from the chain (\ref{eq:generic partition}) by summing $\pi(s):=n-s+1$ to each $p_s$, i.e., $a_s=p_s+\pi(s)$, for $s=1,\ldots,n$.
In terms of generating functions this means multiplying by $(q^s)^{\pi(s)-\pi(s+1)}$ each factor $\frac{1}{1-q^s}$ (we add $\pi(s)$ to the $s$-th part), where $\pi(s)=0$ if $s>n$. Since
\[
\sum_{s=1}^n (s\pi(s)-s\pi(s+1))=\sum_{s=1}^n \pi(s),
\]
this is equivalent to multiply the generating function by $q^{\sum_{s=1}^n \pi(s)}$. In this way we get the known generating function of the partitions with exactly $n$ distinct parts:
\[
\prod_{s=1}^n \frac{(q^s)^{\pi(s)-\pi(s+1)}}{1-q^s}=\prod_{s=1}^n \frac{q^{\pi(s)}}{1-q^s}=\frac{q^{\sum_{s=1}^n \pi(s)}}{\prod_{s=1}^n (1-q^s)}= \frac{q^{\frac{n^2+n}{2}}}{(q)_n}.
\]
Different choices of the function $\pi(s)$ give different conditions on partitions and different generating functions. All the partition theoretic interpretations of the sum side of the analytical identity given in \cite{Hir79}, \cite{Sub85}, \cite{SA88} and in \cite{Cap04} can be described in this way. See the Appendix for more details. 

We introduce the following useful notation: Let  $a,b,r\in\ZZ_{\ge 0}$, when we write $a\ge_r b$ we mean $a-b\ge r$. Therefore, $a\ge_0 b$ means $a \ge b$ and $a\ge_1 b$ means $a > b$. With this notation we can write the chains  (\ref{eq:generic partition}) and (\ref{eq:distinct parts}) as
\begin{align*}
&p_1 \ge_0 p_2\ge_0 \ldots \ge_0 p_{n-1} \ge_0 p_n \ge_0 0, \\
&a_1\ge_1 a_2\ge_1\ldots \ge_1 a_{n-1}\ge_1 a_n\ge_0 1.
\end{align*}
We also define the \emph{parity function} on $\ZZ$:
\[
\parity(m):=\begin{cases}
0, & \text{ if $m$ is even,} \\
1, & \text{ if $m$ is odd.}
\end{cases}
\]

\section{Interpretations of generating functions and partition theorems}\label{sec:results}

We start this section with a partition theoretic interpretation of a quite common kind of generating function.
\begin{proposizione}\label{prop:count}
Let $S, \pi$ and $u$ be functions from $ \ZZ_{\ge 0}$ to $ \ZZ_{\ge 0}$ such that $\pi$ is weakly decreasing and $\sum_{s=1}^{u(n)}\pi(s)=S(n)$. Then
\[
f(q)= \sum_{n=0}^{\infty}\frac{q^{S(n)}}{(q)_{u(n)}},
\]
is the generating function of the partitions satisfying
\[
a_1 \ge_{\pi(1)-\pi(2)} \ldots \ge_{\pi(s-1)-\pi(s)} a_s \ge_{\pi(s)-\pi(s+1)} a_{s+1} \ge_{\pi(s+1)-\pi(s+2)} \ldots \ge_{\pi(n-1)-\pi(n)} a_{u(n)}\ge_0 \pi(n).
\]
\end{proposizione}
\begin{proof}
The proposition follows observing that
\[
\frac{q^{S(n)}}{(q)_{u(n)}}= \frac{q^{\sum_{s=1}^{u(n)} \pi(s)}}{\prod_{s=1}^{u(n)} (1-q^s)}=\prod_{s=1}^{u(n)} \frac{q^{\pi(s)}}{1-q^s}
\]
counts the partitions with at most $u(n)$ parts satisfying the desired chain of inequalities.
\end{proof}
\begin{example}
We consider, for example, the following generating function
\[
f(q)=\sum_{n=0}^{\infty} \frac{q^{n^2+2n}}{(q)_{2n}}.
\]
Here $S(n)=n^2+2n$ and $u(n)=2n$. When $n=3$, we have $S(3)=15$ and $u(3)=6$. The reader can check that the coefficient of $q^{18}$ is $3$, so we expect $3$ partitions satisfying the corresponding inequalities for each choice of the function $\pi$. In \cite{Hir79} the choice for $\pi$ is
\[
\pi(s)=\tfrac{1}{2}(2n+2+\parity(s)-s).
\]
This choice corresponds to count partitions with exactly $2n$ parts such that
\[
a_1 \ge_1 a_2 \ge_0 a_3 \ge_1 a_4 \ge_0 \ldots \ge_1 a_{2n}\ge_0 1.
\]
When $N=18$ and $n=3$ we have only the partitions:
\begin{align*}
 (7, 3, 3, 2, 2, 1), \qquad (6, 4, 3, 2, 2, 1), \qquad (5, 4, 4, 2, 2, 1),
\end{align*}
that satisfy
\[
a_1 \ge_1 a_2 \ge_0 a_3 \ge_1 a_4 \ge_0 a_5 \ge_1 a_6\ge_0 1.
\]
But we can choose $\pi$ differently, for instance we can choose
\[
\pi(s)=\begin{cases}
n^2+1, & \text{ if } s=1, \\
1, & \text{ if } s\ne 1.
\end{cases}
\]
This choice corresponds to count partitions with exactly $2n$ parts such that
\[
a_1 \ge_{n^2} a_2 \ge_0 a_3 \ge_0 a_4 \ge_0 \ldots \ge_0 a_{2n}\ge_0 1.
\]
When $N=18$ and $n=3$ we have only the partitions:
\begin{align*}
 (13, 1, 1, 1, 1, 1), \qquad (12, 2, 1, 1, 1, 1), \qquad (11, 2, 2, 1, 1, 1),
\end{align*}
that satisfy
\[
a_1 \ge_9 a_2 \ge_0 a_3 \ge_0 a_4 \ge_0 a_5 \ge_0 a_6\ge_0 1.
\]
Alternatively, we can choose
\[
\pi(s)=\begin{cases}
n^2+2n, & \text{ if } s=1, \\
0, & \text{ if } s\ne 1,
\end{cases}
\]
and this choice corresponds to count partitions with at most $2n$ parts such that
\[
a_1 \ge_{n^2+2n} a_2 \ge_0 a_3 \ge_0 a_4 \ge_0 \ldots \ge_0 a_{2n}\ge_0 0.
\]
Again, when $N=18$ and $n=3$ we have only the partitions:
\begin{align*}
 (18), \qquad  (17, 1), \qquad  (16, 1, 1),
\end{align*}
that satisfy
\[
a_1 \ge_{15} a_2 \ge_0 a_3 \ge_0 a_4 \ge_0 a_5 \ge_0 a_6\ge_0 0.
\]
\end{example}
We can move one step further.
Up to now, we just took the same $\pi$ for each $n$, but we can actually let $\pi$ vary in function of $n$. This leads us to the following theorem.
\begin{teorema}
Let $S\colon \ZZ_{\ge 0}\to \ZZ_{\ge 0}$ and $u\colon \ZZ_{\ge 0}\to \ZZ_{\ge 0}$ be functions. Let $\pi\colon \ZZ_{\ge 0}^2\to \ZZ_{\ge 0}$  be a two variable function that is weakly decreasing in the second variable and $\sum_{s=1}^{u(n)}\pi(n,s)=S(n)$. Then
\[
f(q)= \sum_{n=0}^{\infty}\frac{q^{S(n)}}{(q)_{u(n)}},
\]
is the generating function of the partitions satisfying
\[
a_1 \ge_{\pi(n,1)-\pi(n,2)} \ldots \ge_{\pi(n,s-1)-\pi(n,s)} a_s \ge_{\pi(n,s)-\pi(n,s+1)} \ldots \ge_{\pi(n,n-1)-\pi(n,n)} a_{u(n)}\ge_0 \pi(n,n).
\]
\end{teorema}
This theorem is an immediate generalization of Proposition \ref{prop:count}. As a direct consequence we have the following corollary.
\begin{corollarioth}
Let $S\colon  \ZZ_{\ge 0}\to  \ZZ_{\ge 0}$ and $u\colon  \ZZ_{\ge 0} \to  \ZZ_{\ge 0}$ be functions. Each Rogers-Ramanujuan identity of the following type
\[
\prod_{\substack{k \text{ subjects to some} \\ \text{congruence conditions}}}\frac{1}{1-q^k}= \sum_{n=0}^{\infty}\frac{q^{S(n)}}{(q)_{u(n)}},
\]
gives infinitely many partition identities. 
\end{corollarioth}
\begin{proof}
This follows from the infinitely many possible choices for the function $\pi(n,s)$.
\end{proof}
We give, as an example, some new interpretations of the second classical Rogers-Ramanujan identity:
\begin{equation}\label{RRid}
\prod_{k\equiv 2,3\bmod 5}\frac{1}{1-q^k}= \sum_{n=0}^{\infty}\frac{q^{n^2+n}}{(q)_{n}}.
\end{equation}
\begin{teorema}
The following sets of partitions have the same numbers of elements:
\begin{align*}
P_1&:=\{\text{Partitions with parts congruent to $2$ or $3$ modulo $5$}\}; \\
P_2&:=\{\text{Partitions with parts greater than $1$ and with difference between adjacent parts at least $2$}\}= \\
&=\{\text{Partitions satisfying } a_1 \ge_2 a_2 \ge_2 \ldots \ge_2 a_{n}\ge_0 2\}; \\
P_3&:=\{\text{Partitions satisfying } a_1 \ge_{n^2} a_2 \ge_0 \ldots \ge_0 a_{n}\ge_0 1\}; \\
P_4&:=\{\text{Partitions satisfying } a_1 \ge_{n^2+n} a_2 \ge_0 \ldots \ge_0 a_{n}\ge_0 0\}; \\
P_5&:=\{\text{Partitions satisfying } a_1 \ge_0 a_2 \ge_0 \ldots \ge_0 a_{n}\ge_0 n+1\}.
\end{align*}
\end{teorema}
\begin{proof}
The set $P_1$ corresponds to the classical interpretation of the product side, the set $P_2$ corresponds to the classical interpretation of the sum side and corresponds, with the notation of Proposition \ref{prop:count}, to the choice $\pi(s)=2(n+1-s),$ for $s=1,\ldots,n.$ \\
The set $P_3$ corresponds to the choice $\pi(s)=\begin{cases}
n^2+1, & \text{ if } s=1, \\
1, & \text{ if } s=2,\ldots,n.
\end{cases}$ \\
The set $P_4$ corresponds to the choice $\pi(s)=\begin{cases}
n^2+n, & \text{ if } s=1, \\
0, & \text{ if } s=2,\ldots,n.
\end{cases}$ \\
The set $P_5$ corresponds to the choice $\pi(s)=n+1,$ for $s=1,\ldots,n.$
\end{proof}
In the previous theorem, we chose just some sets that could be considered "simple" or "natural" among the infinitely many possibilities. The set $P_3$ corresponds to choose the function $\pi$ so to have \emph{exactly} $n$ parts with only one further restriction on the biggest part. The set $P_4$ corresponds to choose the function $\pi$ so to have \emph{at most} $n$ parts with only one further restriction on the biggest part. The set $P_5$ corresponds to choose the function $\pi$ so to have exactly $n$ parts with the least part as biggest as possible. This last choice can be visualized by decomposing an oblong number $n^2+n$ with respect to its $n$ rows made up of $n+1$ elements.

\section{Explicit bijections}\label{sec:bij}

Given a generating function $\sum_{n=0}^{\infty}\frac{q^{S(n)}}{(q)_{u(n)}}$, there is a simple way to obtain a bijection between two sets of partitions corresponding to two different choices of $\pi(n,s)$. Let $(a_1,\ldots,a_{u(n)})$ be a partition corresponding to the choice $\pi_a(n,s)$ and let $(b_1,\ldots,b_{u(n)})$ be a partition corresponding to the choice $\pi_b(n,s)$. Then we have the bijection $B$ given by
\begin{equation}\label{eq:bij}
B(a_s)=a_s-\pi_a(n,s)+\pi_b(n,s)=b_s, \quad \text{ for }s=1,\ldots,u(n).
\end{equation}
This is because $a_s=p_s+\pi_a(n,s)$ and $b_s=p_s+\pi_b(n,s)$, where $p_s$ denote the generic "basic" partition as explained in Section \ref{sec:ineq}.

As an application we show a simple bijection for the second classical Rogers-Ramanujan identity (see Equation \ref{RRid}) that preserves the number of parts (but not the weight) of the partitions. If $(a_1,\ldots,a_n)$ is a partition that satisfies the product side conditions, then we want a partition $(b_1,\ldots,b_n)$ that satisfies the sum side classical conditions. We obtain this in two steps. First we give a bijection between a partition $(a_1,\ldots,a_n)$ that satisfies the product side conditions and a partition $(c_1,\ldots,c_n)$ that satisfies the sum side condition associated to the choice
\begin{align*}
\pi_1(s)&:=\begin{cases}
n^2+1, & \text{ if } s=1, \\
1, & \text{ if } s=2,\ldots,n,
\end{cases}
\end{align*}
i.e., $(c_1,\ldots,c_n)$ is a partition that satisfies the following chain condition:
\[
c_1 \ge_{n^2} c_2\ge_0 \ldots \ge_0 c_n\ge_0 1.
\]
This first bijection is given by
\[
c_s=C(a_s):=a_s-3q_s-1+n^2\delta(1,s), \quad \text{ for }s=1,\ldots,n,
\]
where $\delta$ is the Kronecker delta and $q_s:=\left\lfloor\frac{a_s}{5} \right\rfloor$. The second step goes from a partition $(c_1,\ldots,c_n)$ that satisfies the sum side condition associated to $\pi_1(s)$ to a partition $(b_1,\ldots,b_n)$ that satisfies the sum side condition associated to $\pi_c(s)$, where
\begin{align*}
\pi_c(s):=2(n+1-s), \quad s=1,\ldots,n,
\end{align*}
i.e., $(b_1,\ldots,b_n)$ satisfies the chain condition
\[
b_1 \ge_2 b_2\ge_2 \ldots \ge_2 b_n\ge_0 2,
\]
that represents the classical Rogers-Ramanujan conditions. This can be done using Equation~(\ref{eq:bij}):
\[
b_s=B(c_s):=c_s-\pi_1(n,s)+\pi_c(n,s), \quad \text{ for }s=1,\ldots,n.
\]
The composition $I:=B\circ C$ of these two bijections gives the desired one:
\[
b_s=I(a_s):=a_s-3q_s-1+n^2\delta(1,s)-\pi_1(s)+\pi_c(s), \quad \text{ for }s=1,\ldots,n.
\]
The inverse function is given by
\[
a_s=I^{-1}(b_s)=b_s+3k_s+1-n^2\delta(1,s)-\pi_c(s)+\pi_1(s), \quad \text{ for }s=1,\ldots,n,
\]
where $k_s:=\left\lfloor\frac{b_s-1-n^2\delta(1,s)-\pi_c(s)+\pi_1(s)}{2} \right\rfloor=q_s$. We leave to the reader the verification of the details. Both the partitions $(a_1,\ldots,a_n)$ and $(b_1,\ldots,b_n)$ have $n$ parts, but if $N_a$ is the weight of the former partition and $N_b$ is the weight of the latter partition, their relation is
\[
N_b=N_a+n^2-\sum_{s=1}^n(3q_s+1).
\]

\section{Glaisher's identities}\label{sec:euler}

There is a simple explicit bijection between the set of the partitions in distinct parts of a positive integer $N$, that we denote by $\mathcal{D}_N$, and the set of the partitions of $N$ into odd parts, that we denote by $\mathcal{C}_N$. Now we describe the bijection $b\colon \mathcal{D}_N\to \mathcal{C}_N$ given by Euler. Let $\lambda\in \mathcal{D}_N$, as first step we divide by $2$ each even part, then we repeat this step on the resulting partition until all the parts are odd. For example let $N=20$ and $\lambda=(7,6,4,2,1)\in \mathcal{D}_N$, then at the first step we get $(7,3,3,2,2,1,1,1)$. Then, since we still have even parts, we repeat this step obtaining $(7,3,3,1,1,1,1,1,1,1)$. Since we have only odd parts, we are done and $b(\lambda)=(7,3,3,1,1,1,1,1,1,1)$. The inverse map $b^{-1}$ is given, for $\mu\in \mathcal{C}_N$, first pairing the equal parts and summing the elements of each pair, and then  iterating this process on the resulting partition until there are no more equal parts.
For example if $N=20$ and $\mu=(7,3,3,1,1,1,1,1,1,1)$, we sum up pairs of equal parts obtaining $(7,6,2,2,2,1)$. Since we still have pairs of equal parts, we repeat the process and we get $(7,6,4,2,1)$. Now we stop because we have not any pair of equal parts, so $b^{-1}(\mu)=(7,6,4,2,1)$. The Euler's identity:
\[
\prod_{k\equiv 1\bmod 2}\frac{1}{1-q^k}= \sum_{n=0}^{\infty}\frac{q^{\frac{n^2+n}{2}}}{(q)_{n}},
\]
follows by the bijection $b$ described above once checked that the two sides are the generating functions of the two sets $\mathcal{C}_N$ and $\mathcal{D}_N$.
\begin{remark}
We can use the point of view of Section \ref{sec:results} to interpret the right-hand side as the generating function of the set $\mathcal{D}_N$ by choosing the function $\pi_E(s):=n+1-s$, for $s=1,\ldots,n$, or as the generating function of infinitely many other suitable sets. For example we can observe that the exponent $\frac{1}{2}(n^2+n)$ represents the triangular numbers and the choice of $\pi(s)=\pi_E(s)$, can be viewed as the sum of the rows of the triangle associated to the triangular number. But if we take $\pi(s)=\pi_A(s)$, where
\[
\pi_A(s):=\begin{cases}
(n+2-2s)(n+1-2s), & \text{ if } s=1,\ldots,\left\lfloor\frac{n}{2}\right\rfloor, \\
1, & \text{ if $n$ is odd and } s=\frac{n+1}{2}, \\
0, & \text{ if } s=\left\lfloor\frac{n+1}{2}\right\rfloor+1,\ldots,n,
\end{cases}
\]
it corresponds to the sum of each "upper layer" of the triangle (see Figure~\ref{fig:tr}), taking zero for the remaining parts not belonging to any layer.
\end{remark}
\begin{figure}[h!]
\centering
\begin{tikzpicture}[x=1cm,y=0.4cm]
\draw[dashed] (0.15,-1.5)--(2,6);
\draw[dashed] (2,6)--(3.8,-1.5);
\draw[dashed] (-0.6,-0.2)--(2,10);
\draw[dashed] (2,10)--(4.6,-0.4);
\draw[rotate=50,dashed] (-0.5,1)  arc (90:290:0.4cm and 0.45cm); 
\draw[rotate=-60,dashed] (2.45,7.5)  arc (270:460:0.4cm and 0.45cm); 
\foreach \Point in {(0,0), (1,0), (2,0),  (3,0),  (4,0), (0.5,2), (1.5,2),  (2.5,2),  (3.5,2), (1,4), (2,4), (3,4), (1.5,6), (2.5,6), (2,8)}{    \node at \Point {\textbullet};}
\end{tikzpicture}
\caption{First upper layer for $N=20$.} \label{fig:tr}
\end{figure}
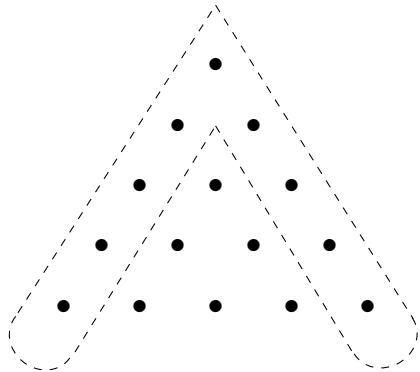
\begin{example}
Let $N=20$ and $\lambda=(7,6,4,2,1)$ which, in the classical bijection of Euler corresponds to $(7,3,3,1,1,1,1,1,1,1)$. Subtract from $\lambda$ the staircase $(5,4,3,2,1)$  to get $(2,2,1,0,0)$. Now add the "layers" of the triangle: $(9,5,1,0,0)$. The result is $(11,7,2,0,0)$. So the partition $(11,7,2)$ corresponds to $(7,3,3,1,1,1,1,1,1,1)$. We do not see a natural way for obtaining a bijection between these new types of partitions. 
\end{example}
The Euler's bijection has been generalized by Glaisher, in \cite{Gla83}, to a generic integer greater than 1. In the following theorem we recall the result of Glaisher and then we show corresponding Rogers-Ramanujan type identities.
\begin{teorema}
Let $M\in\ZZ_{\ge 2}$, let $\mathcal{D}_N(M)$ be the set of the partitions of $N$ in parts such that each part is repeated strictly less than $M$ times and let $\mathcal{C}_N(M)$ be the set of the partitions of $N$ with parts not congruent to $0$ modulo $M$. 
\begin{enumerate}
\item\label{eulergenbij} Let $b_M\colon \mathcal{D}_N(M)\to \mathcal{C}_N(M)$ be defined as follows: for every $\lambda\in \mathcal{D}_N(M)$, we divide by $M$ each part of $\lambda$ divisible by $M$ and we iterate this step until all the resulting parts are not divisible by $M$ anymore. Then $b_M$ is a bijection and its inverse $b_M^{-1}$ is given, for $\mu\in \mathcal{C}_N(M)$, summing up all the $M$-tuples of equal parts of $\mu$ and iterating this process until there are no more $M$ equal parts. 
\item\label{eulergenanid} The following identity holds:
\[
\prod_{k\nequiv 0\bmod M}\frac{1}{1-q^k}= 1+ \sum_{n=1}^{\infty}\frac{(q^{n}-q^{nM})(q^M)_{n-1}}{(q)_n}.
\]
\end{enumerate}
\end{teorema}
\begin{proof}
Part (\ref{eulergenbij}). See \cite{Gla83}. \\
Part (\ref{eulergenanid}). The product in the left-hand side is the generating function of the partitions in $\mathcal{C}_N(M)$. The conjugate partition of $\lambda\in \mathcal{D}_N(M)$ is a partition of $N=a_1+\ldots+a_n$ such that
\[
a_1\ge_0^{M-1}\ldots\ge_0^{M-1}a_n\ge_1^{M-1}0,
\]
where $a\ge_r^s b$, for $r,s\in\ZZ_{\ge 0}$ and $r\le s$, means $r\le a-b \le s$. We claim that the second sum is the generating function of these partitions, i.e., of the conjugate partitions of the partitions inside $\mathcal{D}_N(M)$. Hence Part (\ref{eulergenbij}) implies the equality of the generating functions. Now we prove the claim. We use the functional equation method. Let
\[
f(z,q)=\sum_{n=0}^{\infty}\sum_{N=0}^{\infty}g(n,N)q^N z^n =\sum_{n=0}^{\infty}\alpha_n(q)z^n,
\]
where $g(n,N)$ is the cardinality of the set $\mathcal{G}(n,N)$ of the partitions of $N$ with $n$ parts satisfying the conditions $a_1\ge_0^{M-1}\ldots\ge_0^{M-1}a_n\ge_1^{M-1}0.$ Taking the subset $L(s,j;n,N)$ of $\mathcal{G}(n,N)$ of the partitions with $s$ least parts equal to $j$, we have the bijection "delete all the least parts and subtract $j$ to the others" between $L(s,j;n,N)$ and $\mathcal{G}(n-s,N-jn)$.
Hence, we have
\begin{align*}
g(n,N)&=\underbrace{\underbrace{g(n-1,N-n)}_{\text{only $1$ part equals to }1}+\underbrace{g(n-2,N-n)}_{2\text{ parts equal to }1}+\ldots+\underbrace{g(1,N-n)}_{n-1\text{ parts equal to }1}+\underbrace{g(0,N-n)}_{n\text{ parts equal to }1}}_{=\sum_{s=1}^n g(n-s,N-n), \quad \text{least parts equal to }1}+ \\
&+\underbrace{\sum_{s=1}^n g(n-s,N-2n)}_{\text{least parts equal to }2}+\underbrace{\sum_{s=1}^n g(n-s,N-3n)}_{\text{least parts equal to }3}+\ldots+\underbrace{\sum_{s=1}^n g(n-s,N-(M-1)n)}_{\text{least parts equal to }M-1}= \\
&=\sum_{j=1}^{M-1}\sum_{s=1}^n g(n-s,N-jn).
\end{align*}
It follows that
\begin{align*}
f(z,q)&=\sum_{n=0}^{\infty}\sum_{j=1}^{M-1}\sum_{s=1}^n \sum_{N=0}^{\infty} g(n-s,N-jn)q^N z^n=\sum_{n=0}^{\infty}\sum_{j=1}^{M-1}\sum_{s=1}^n \sum_{N=0}^{\infty} g(n-s,N)q^{N+jn} z^n= \\
&=\sum_{n=0}^{\infty}\sum_{j=1}^{M-1}\sum_{s=1}^n \alpha_{n-s}(q)q^{jn} z^n=\sum_{n=0}^{\infty}\sum_{j=1}^{M-1}q^{jn} \sum_{s=1}^n \alpha_{n-s}(q)z^n=\sum_{n=0}^{\infty}\frac{q^n-q^{nM}}{1-q^n} \sum_{s=1}^n \alpha_{n-s}(q)z^n,
\end{align*}
and comparing the coefficients of $z^n$ we get
\begin{equation}\label{receq}
\alpha_n(q)=\frac{q^n-q^{nM}}{1-q^n} \sum_{s=1}^n \alpha_{n-s}(q).
\end{equation}
Using that $\alpha_0(q)=1$, one can prove by induction on $n$ that
\[
\alpha_n(q)=\frac{q^{n}-q^{nM}}{1-q^n}\prod_{j=1}^{n-1}\frac{1-q^{jM}}{1-q^j},
\]
satisfies the recurrence condition (\ref{receq}). Hence, $f(1,q)$ gives the claim.
\end{proof}
\begin{remark}
When $M=2$, we get the classical Euler's bijection and the following identity:
\[
\prod_{k\nequiv 0\bmod 2}\frac{1}{1-q^k}= 1+\sum_{n=1}^{\infty}\frac{q^{\frac{n}{2}(n+1)}}{(q)_n}
= 1+ \sum_{n=1}^{\infty}q^{n}\prod_{j=1}^{n-1}(1+q^{j}),
\]
where the first equality is the classical Euler's identity.
\end{remark}

\section*{Acknowledgements}

I would like to thank Andrea Vietri and Alberto Del Fra for their remarks and suggestions and in particular Stefano Capparelli that introduced me to this topic and gave me many useful suggestions.

\section*{Appendix}\label{sec:oldex}

We write some known results contained in \cite{Cap04}, \cite{Hir79}, \cite{SA88} and \cite{Sub85}, using the notation introduced in Section~\ref{sec:ineq}. We notice that there are bijections between: cases \ref{it:euler} and \ref{it:subagr3}, cases \ref{it:capparelli1} and \ref{it:subagr1}, cases \ref{it:hirschhorn1} and \ref{it:subbarao2}, cases \ref{it:hirschhorn2} and \ref{it:subbarao1}, cases \ref{it:hirschhorn3} and \ref{it:subagr2} and \ref{it:subbarao4}, cases \ref{it:hirschhorn4} and \ref{it:subbarao3}.
\begin{enumerate}[a)]
\item\label{it:euler} A classical result of Euler is: The number of partitions satisfying
\begin{equation}\label{eq:euler}
a_1\ge_1 a_2\ge_1 a_3\ge_1\ldots \ge_1 a_{n-1}\ge_1 a_n\ge_0 1,
\end{equation}
is the same number as the number of partitions whose parts are congruent to $1 \pmod 2$. We get (\ref{eq:euler}) from (\ref{eq:generic partition}) choosing: $\pi(s)=n+1-s,$ and $S(n)=\tfrac{1}{2}(n^2+n).$

\item\label{it:capparelli1} \cite[Theorem 1.6]{Cap04}: The number of partitions satisfying
\begin{align}\label{eq:capparelli1}
a_1 \ge_0 a_2 \ge_1 a_3 \ge_0 a_4 \ge_1  \ldots  \ge_{\parity(m-1)}  a_{m-1} \ge_0 a_m\ge_0 2,
\end{align}
is the same number as the number of partitions whose parts are congruent to \\ $\pm 2,\pm 3,\pm 4,\pm 5, \pm 6, \pm 7 \pmod{20}$.  We get (\ref{eq:capparelli1}) from (\ref{eq:generic partition}) choosing:
\begin{align*}
\text{when } &m=2n \text{ is even} &  \text{when } &m=2n-1 \text{ is odd}  \\
\pi(s)&=\tfrac{1}{2}(2n+4-\parity(s)-s), & \pi(s)&=\begin{cases}
\tfrac{1}{2}(2n+2-\parity(s)-s), & \text{ if } s<2n-1, \\
2, & \text{ if }s=2n-1,
\end{cases} \\
S(n)&=n^2+3n, & S(n)&=n^2+n.
\end{align*}

\item\label{it:capparelli2} \cite[Theorem 1.7]{Cap04}: The number of partitions satisfying
\begin{equation}\label{eq:capparelli2}
a_1 \ge_1 a_2 \ge_0 a_3 \ge_1 a_4 \ge_0  \ldots  \ge_{\parity(m-2)}  a_{m-1} \ge_1 a_m\ge_0 1,
\end{equation}
is the same number as the number of partitions whose parts are congruent to \\ $\pm 1,\pm 3,\pm 5, \pm 7,\pm 8,\pm 9 \pmod{20}$. We get (\ref{eq:capparelli2}) from (\ref{eq:generic partition}) choosing:
\begin{align*}
\text{when } &m=2n \text{ is even} &  \text{when } &m=2n-1 \text{ is odd}  \\
\pi(s)&=\tfrac{1}{2}(2n+2+\parity(s)-s), & \pi(s)&=\begin{cases}
\tfrac{1}{2}(2n+2+\parity(s)-s), & \text{ if }s<2n-1, \\
1, & \text{ if }s=2n-1,
\end{cases} \\
S(n)&=n^2+2n, & S(n)&=n^2+2n-2.
\end{align*}

\item\label{it:hirschhorn1} \cite[Theorem 1]{Hir79}: The number of partitions satisfying
\begin{equation}\label{eq:hirschhorn1}
a_1 \ge_1 a_2 \ge_0 a_3 \ge_1 a_4 \ge_0 \ldots \ge_{\parity(m-1)} a_m\ge_0 1,
\end{equation}
is the same number as the number of partitions whose parts are congruent to \\ $\pm 1,\pm 3,\pm 4, \pm 5, \pm 7, \pm 9 \pmod{20}$. We get (\ref{eq:hirschhorn1}) from (\ref{eq:generic partition}) choosing:
\begin{align*}
\text{when } &m=2n \text{ is even} &  \text{when } &m=2n-1 \text{ is odd}  \\
\pi(s)&=\tfrac{1}{2}(2n+2+\parity(s)-s), & \pi(s)&=\tfrac{1}{2}(2n+\parity(s)-s), \\
S(n)&=n^2+2n, & S(n)&=n^2.
\end{align*}

\item\label{it:hirschhorn2} \cite[Theorem 2]{Hir79}: The number of partitions satisfying
\begin{equation}\label{eq:hirschhorn2}
a_1 \ge_0 a_2 \ge_1 a_3 \ge_0 a_4 \ge_1 \ldots  \ge_{\parity(m)} a_m\ge_0 1,
\end{equation}
is the same number as the number of partitions whose parts are congruent to \\ $\pm 1,\pm 2,\pm 5, \pm 6,\pm 8,\pm 9 \pmod{20}$. We get (\ref{eq:hirschhorn2}) from (\ref{eq:generic partition}) choosing:
\begin{align*}
\text{when } &m=2n \text{ is even} &  \text{when } &m=2n-1 \text{ is odd}  \\
\pi(s)&=\tfrac{1}{2}(2n+2-\parity(s)-s), & \pi(s)&=\tfrac{1}{2}(2n+2-\parity(s)-s), \\
S(n)&=n^2+n, & S(n)&=n^2+n-1.
\end{align*}

\item\label{it:hirschhorn3} \cite[Theorem 3]{Hir79}: The number of partitions satisfying
\begin{equation}\label{eq:hirschhorn3}
a_1 \ge_{2\parity(m-1)} a_2 \ge_{2\parity(m-2)} a_3 \ge_{2\parity(m-3)} \ldots \ge_0  a_{m-1} \ge_2 a_m\ge_0 1,
\end{equation}
is the same number as the number of partitions whose parts are congruent to \\ $\pm 1,\pm 4,\pm 6, \pm 7 \pmod{16}$. We get (\ref{eq:hirschhorn3}) from (\ref{eq:generic partition}) choosing:
\begin{align*}
\text{when } &m=2n \text{ is even} &  \text{when } &m=2n-1 \text{ is odd}  \\
\pi(s)&=2n+1+\parity(s)-s, & \pi(s)&=2n+1-\parity(s)-s, \\
S(n)&=2n^2+2n, & S(n)&=2n^2-1.
\end{align*}

\item\label{it:hirschhorn4} \cite[Theorem 4]{Hir79}: The number of partitions satisfying
\begin{equation}\label{eq:hirschhorn4}
a_1 \ge_{2\parity(m)} a_2 \ge_{2\parity(m-1)} a_3 \ge_{2\parity(m-2)}  \ldots  \ge_2  a_{m-1} \ge_0 a_m\ge_2 0,
\end{equation}
is the same number as the number of partitions whose parts are congruent to \\ $\pm 2,\pm 3,\pm 4,\pm 5 \pmod{16}$. We get (\ref{eq:hirschhorn4}) from (\ref{eq:generic partition}) choosing:
\begin{align*}
\text{when } &m=2n \text{ is even} &  \text{when } &m=2n-1 \text{ is odd}  \\
\pi(s)&=2n+2-\parity(s)-s, & \pi(s)&=2n+\parity(s)-s, \\
S(n)&=2n^2+2n, & S(n)&=2n^2.
\end{align*}

\item\label{it:subagr1} \cite[Theorem 1.4]{SA88}: The number of partitions satisfying
\begin{equation}\label{eq:subagr1}
a_1 \ge_2 a_2 \ge_2 a_3 \ge_2 \ldots \ge_2  a_{\left\lfloor \frac{m}{2}\right\rfloor} \ge_0 \ldots  \ge_0  a_{m-1} \ge_0 a_m\ge_0 2,
\end{equation}
is the same number as the number of partitions whose parts are congruent to \\ $\pm 2,\pm 3,\pm 4,\pm 5, \pm 6, \pm 7 \pmod{20}$. We get (\ref{eq:subagr1}) from (\ref{eq:generic partition}) choosing:
\begin{align*}
\text{when } &m=2n \text{ is even} &  \text{when } &m=2n-1 \text{ is odd}  \\
\pi(s)&=\begin{cases}
2(n+1-s), & \text{ if } s<n, \\
2, & \text{ if } s\ge n,
\end{cases} & \pi(s)&=\begin{cases}
2(n-s), & \text{ if } s<n, \\
2, & \text{ if } s\ge n,
\end{cases} \\
S(n)&=n^2+3n, & S(n)&=n^2+n.
\end{align*}

\item\label{it:subagr2} \cite[Theorem 1.5]{SA88}: The number of partitions satisfying
\begin{equation}\label{eq:subagr2}
a_1 \ge_2 a_2 \ge_2 a_3 \ge_2  \ldots \ge_2 a_n \ge_1 a_{n+1} \ge_{n-1} a_{n+2} \ge_0 \ldots  \ge_0  a_{2n} \ge_0 a_{2n+1}\ge_0 1,
\end{equation}
is the same number as the number of partitions whose parts are congruent to \\ $\pm 1,\pm 4,\pm 6, \pm 7 \pmod{16}$. In this case, we have $m=2n+1$ odd, we get (\ref{eq:subagr2}) from (\ref{eq:generic partition}) choosing:
\[
\pi(s)=\begin{cases}
3n+1-2s, & \text{ if } s< n+1, \\
n, & \text{ if } s= n+1, \\
1, & \text{ if } s> n+1,
\end{cases} \qquad S(n)=2n^2+2n.
\]

\item\label{it:subagr3} \cite[Theorem 1.6]{SA88}: The number of partitions satisfying
\begin{equation}\label{eq:subagr3}
a_1 \ge_2 a_2 \ge_2 a_3 \ge_2 \ldots \ge_2 a_n \ge_0 a_{n+1} \ge_{n-1} a_{n+2} \ge_0 \ldots  \ge_0  a_{2n} \ge_0 a_{2n+1}\ge_0 1,
\end{equation}
is the same number as the number of partitions whose parts are congruent to $1 \pmod 2$. In this case, we have $m=2n+1$ odd, we get (\ref{eq:subagr3}) from (\ref{eq:generic partition}) choosing:
\[
\pi(s)=\begin{cases}
3n-2s, & \text{ if } s< n+1, \\
n, & \text{ if } s= n+1, \\
1, & \text{ if } s> n+1,
\end{cases} \qquad S(n)=2n^2+n.
\]

\item\label{it:subbarao1} \cite[Theorem 2.1]{Sub85}: The number of partitions satisfying
\begin{equation}\label{eq:subbarao1}
a_1 \ge_2 a_2 \ge_2 a_3 \ge_2 \ldots \ge_2  a_{\left\lfloor \frac{m+1}{2}\right\rfloor} \ge_0 \ldots  \ge_0  a_{m-1} \ge_0 a_m\ge_0 1,
\end{equation}
is the same number as the number of partitions whose parts are congruent to \\ $\pm 1, \pm 2, \pm 5,\pm 6, \pm 8, \pm 9 \pmod{20}$. We get (\ref{eq:subbarao1}) from (\ref{eq:generic partition}) choosing:
\begin{align*}
\text{when } &m=2n \text{ is even} &  \text{when } &m=2n-1 \text{ is odd}  \\
\pi(s)&=\begin{cases}
2n+1-2s, & \text{ if } s<n, \\
1, & \text{ if } s\ge n,
\end{cases} & \pi(s)&=\begin{cases}
2n+1-2s, & \text{ if } s<n, \\
1, & \text{ if } s\ge n,
\end{cases} \\
S(n)&=n^2+n, & S(n)&=n^2+n-1.
\end{align*}

\item\label{it:subbarao2} \cite[Theorem 2.2]{Sub85}: The number of partitions satisfying
\begin{equation}\label{eq:subbarao2}
a_1 \ge_2 a_2 \ge_2 a_3 \ge_2 \ldots \ge_2  a_{\left\lfloor \frac{m}{2}\right\rfloor}\ge_1  a_{\left\lfloor \frac{m+2}{2}\right\rfloor} \ge_0 \ldots  \ge_0  a_{m-1} \ge_0 a_m\ge_0 1,
\end{equation}
is the same number as the number of partitions whose parts are congruent to \\ $\pm 1,\pm 3, \pm 4,\pm 5, \pm 7, \pm 9 \pmod{20}$. We get (\ref{eq:subbarao2}) from (\ref{eq:generic partition}) choosing:
\begin{align*}
\text{when } &m=2n \text{ is even} &  \text{when } &m=2n-1 \text{ is odd}  \\
\pi(s)&=\begin{cases}
2n+2-2s, & \text{ if } s\le n, \\
1, & \text{ if } s> n,
\end{cases} & \pi(s)&=\begin{cases}
2n-2s, & \text{ if } s<n, \\
1, & \text{ if } s\ge n,
\end{cases} \\
S(n)&=n^2+2n, & S(n)&=n^2.
\end{align*}

\item\label{it:subbarao3} \cite[Theorem 2.3]{Sub85}: The number of partitions satisfying
\begin{equation}\label{eq:subbarao3}
a_1 \ge_2 a_2 \ge_2 a_3 \ge_2 \ldots \ge_2  a_{\left\lfloor \frac{m}{2}\right\rfloor} \ge_{\left\lfloor \frac{m}{2}\right\rfloor+(-1)^{m-1}}  a_{\left\lfloor \frac{m+2}{2}\right\rfloor} \ge_0 \ldots  \ge_0  a_{m-1} \ge_0 a_m\ge_0 2,
\end{equation}
is the same number as the number of partitions whose parts are congruent to \\ $\pm 2,\pm 3,\pm 4, \pm 5 \pmod{16}$. We get (\ref{eq:subbarao3}) from (\ref{eq:generic partition}) choosing:
\begin{align*}
\text{when } &m=2n \text{ is even} &  \text{when } &m=2n-1 \text{ is odd}  \\
\pi(s)&=\begin{cases}
3n+1-2s, & \text{ if } s\le n, \\
2, & \text{ if } s> n,
\end{cases} & \pi(s)&=\begin{cases}
3n-2s, & \text{ if } s<n, \\
2, & \text{ if } s\ge n,
\end{cases} \\
S(n)&=2n^2+2n, & S(n)&=2n^2.
\end{align*}

\item\label{it:subbarao4} \cite[Theorem 2.4]{Sub85}: The number of partitions satisfying
\begin{equation}\label{eq:subbarao4}
a_1 \ge_2 a_2 \ge_2 \ldots \ge_2  a_{\left\lfloor \frac{m-1}{2}\right\rfloor} \ge_{2-\parity(m)}  a_{\left\lfloor \frac{m+1}{2}\right\rfloor} \ge_{n-1-\parity(m)}  a_{\left\lfloor \frac{m+3}{2}\right\rfloor} \ge_0 \ldots  \ge_0  a_{m-1} \ge_0 a_m\ge_0 2,
\end{equation}
is the same number as the number of partitions whose parts are congruent to \\ $\pm 1,\pm 4,\pm 6, \pm 7 \pmod{16}$. We get (\ref{eq:subbarao4}) from (\ref{eq:generic partition}) choosing:
\begin{align*}
\text{when } &m=2n \text{ is even} &  \text{when } &m=2n-1 \text{ is odd}  \\
\pi(s)&=\begin{cases}
3n+1-2s, & \text{ if } s\le n, \\
2, & \text{ if } s> n,
\end{cases} & \pi(s)&=\begin{cases}
3n-1-2s, & \text{ if } s< n, \\
n, & \text{ if } s= n, \\
2, & \text{ if } s> n,
\end{cases} \\
S(n)&=2n^2+2n, & S(n)&=2n^2-1.
\end{align*}

\end{enumerate}

\end{document}